\documentclass[11pt]{article}
\usepackage{amsmath,amsthm,amssymb,stackengine}
\usepackage[hyperfootnotes=false]{hyperref}
\usepackage{color}
\usepackage{cancel}
\usepackage{titlesec}
\titleformat{\paragraph}
{\normalfont\normalsize\bfseries}{\theparagraph}{1em}{}
\titlespacing*{\paragraph}
{0pt}{3.25ex plus 1ex minus .2ex}{1.5ex plus .2ex}

\def\titlerunning#1{\gdef\titrun{#1}}
\makeatletter
\def\author#1{\gdef\autrun{\def\and{\unskip, }#1}\gdef\@author{#1}}
\def\address#1{{\def\and{\\\hspace*{18pt}}\renewcommand{\thefootnote}{}%
\footnote {#1}}%
\markboth{\autrun}{\titrun}}
\makeatother
\def\email#1{\hspace*{4pt}{\em e-mail}: #1}
\def\MSC#1{{\renewcommand{\thefootnote}{}%
\footnote{\emph{Mathematics Subject Classification (2020):} #1}}}
\def\keywords#1{\par\medskip
\noindent\textbf{Keywords:} #1}

\newtheorem{theorem}{Theorem}[section]

\newtheorem{prop}[theorem]{Proposition}
\newtheorem{cor}[theorem]{Corollary}
\newtheorem{lemma}[theorem]{Lemma}

\theoremstyle{definition}

\newtheorem{remark}[theorem]{Remark}

\numberwithin{equation}{section}

\def\cA{\mathcal A}

\def\cC{\mathcal C}
\def\cD{\mathcal D}
\def\cE{\mathcal E}

\def\cG{\mathcal G}
\def\cH{\mathcal H}

\def\cL{\mathcal L}

\def\cO{\mathcal O}
\def\cP{\mathcal P}
\def\cQ{\mathcal Q}

\def\cR{\mathcal R}
\def\cS{\mathcal S}
\def\cT{\mathcal T}

\def\cW{\mathcal W}
\def\cX{\mathcal X}

\def\PG{{\rm PG}}

\def\F{{\mathbb F}}

\def\PGL{{\rm PGL}}

\def\PSL{{\rm PSL}}
\def\PSp{{\rm PSp}}

\def\j{\boldsymbol j}

\def\v{\boldsymbol v}

\def\w{\omega}

\begin{document}


\baselineskip=16pt

\titlerunning{}

\title{The $m$-ovoids of $\cW(5, 2)$}

\author{Michela Ceria
\and 
Francesco Pavese}

\date{}

\maketitle

\address{M. Ceria, F. Pavese: Dipartimento di Meccanica, Matematica e Management, Politecnico di Bari, Via Orabona 4, 70125 Bari, Italy; \email{\{michela.ceria, francesco.pavese\}@poliba.it}
}


\MSC{Primary 51E22; 94B05. Secondary 51E20.}

\begin{abstract}
In this paper we are concerned with $m$-ovoids of the symplectic polar space $\cW(2n+1, q)$, $q$ even. In particular we show the existence of an elliptic quadric of $\PG(2n+1, q)$ not polarizing to $\cW(2n+1, q)$ forming a $\left(\frac{q^n-1}{q-1}\right)$-ovoid of $\cW(2n+1, q)$. A further class of $(q+1)$-ovoids of $\cW(5, q)$ is exhibited. It arises by glueing together two orbits of a subgroup of $\PSp(6, q)$ isomorphic to $\PSL(2, q^2)$. We also show that the obtained $m$-ovoids do not fall in any of the examples known so far in the literature. Moreover, a computer classification of the $m$-ovoids of $\cW(5, 2)$ is acquired. It turns out that $\cW(5, 2)$ has $m$-ovoids if and only if $m = 3$ and that there are exactly three pairwise non-isomorphic examples. The first example comes from an elliptic quadric $\cQ^-(5, 2)$ polarizing to $\cW(5, 2)$, whereas the other two are the $3$-ovoids previously mentioned. 

\end{abstract}

\section{Introduction}

Let $V$ be a $2(n+1)$-dimensional vector space over $\F_q$, where $\F_q$ is the finite field of order $q$ and let $\PG(V) = \PG(2n+1, q)$ be the projective space whose underlying vector space is $V$. We will use the term $r$-space to indicate an $r$-dimensional projective subspace of $\PG(V)$. The symplectic polar space $\cW(2n+1, q)$ consists of the subspaces of $\PG(2n+1, q)$ induced by the totally isotropic subspaces of $V$ with respect to a non-degenerate alternating form of $V$. The subspaces of $\cW(2n+1, q)$ of maximum dimension are $n$-spaces and are called {\em generators}. An $m$-ovoid $\cX$ of $\cW(2n+1, q)$ is a set of points meeting every generator of $\cW(2n+1, q)$ in $m$ points. The case $m = 1$ was studied by J. Thas who proved that $\cW(2n+1, q)$ has a $1$-ovoid if and only if $n = 1$ and $q$ is even \cite{Thas}. An $h$-system of $\cW(2n+1, q)$ is a set of $q^{n+1}+1$ pairwise opposite $h$-spaces of $\cW(2n+1, q)$, see \cite{ST} for more details. The concept of $m$-ovoids first appeared in \cite{ST}, where the authors proved that the set of points covered by an $h$-system of $\cW(2n+1, q)$ (and more generally of a finite classical polar space) is a $\left(\frac{q^{h+1}-1}{q-1}\right)$-ovoid of $\cW(2n+1, q)$. These objects were further investigated in \cite{BKLP} where the following equivalent definition of an $m$-ovoid $\cX$ of $\cW(2n+1, q)$ was given
\begin{align*}
|\cX \cap P^\perp| = 
\begin{cases}
m(q^{n-1}+1) - q^{n-1} & \mbox{ if } P \in \cX, \\
m(q^{n-1} +1) & \mbox{ if } P \notin \cX. 
\end{cases}
\end{align*} 
It follows that an $m$-ovoid $\cX$ of $\cW(2n+1, q)$ has two intersection numbers with respect to hyperplanes and hence gives rise to a strongly regular graph and a projective $2$-weight code \cite{CK}. As for constructions, field reduction provides a method for obtaining an $\left(m \frac{q^{2t}-1}{q-1}\right)$-ovoid of $\cW(2t(2n+1)-1, q)$ from an $m$-ovoid of a Hermitian polar space $\cH(2n, q^{2t})$  and an $\left(m \frac{q^{t}-1}{q-1}\right)$-ovoid of $\cW(2nt-1, q)$ from an $m$-ovoid of $\cW(2n-1, q^t)$, see \cite{K}. Another technique for constructing $m$-ovoids of $\cW(2n+1, q)$ has been recently outlined in \cite{FWX} by means of certain strongly regular Cayley graphs over $(\F_q^{2n+2}, +)$ of negative Latin square type. This allows the authors to obtain $m$-ovoids of $\cW(2n+1, q)$, $q$ odd, for many values of $m$, see \cite[Theorem 3.5]{FWX}. Further constructions can be found in \cite{BLP, CCEM, CP, CP1}. 

Assume that $q$ is even and let $\cQ^-(2n+1, q)$ be an elliptic quadric of $\PG(2n+1, q)$. Since $q$ is even, the quadric $\cQ^-(2n+1, q)$ defines a symplectic polarity of $\PG(2n+1, q)$. If $\cW(2n+1, q)$ denotes the associated symplectic polar space of $\PG(2n+1, q)$, then it is said that $\cQ^-(2n+1,q )$ is an elliptic quadric {\em polarizing to} $\cW(2n+1, q)$. In this case $\cQ^-(2n+1, q)$ is a $\left(\frac{q^n-1}{q-1}\right)$-ovoid of $\cW(2n+1, q)$. In Subsection~\ref{FirstConstr}, we show the existence of an elliptic quadric of $\PG(2n+1, q)$ not polarizing to $\cW(2n+1, q)$ forming a $\left(\frac{q^n-1}{q-1}\right)$-ovoid of $\cW(2n+1, q)$, see Theorem~\ref{first}). In Subsection~\ref{SecondConstr}, a further class of $(q+1)$-ovoids of $\cW(5, q)$, $q$ even, is presented, see Theorem~\ref{second}. In particular, it is obtained by glueing together two orbits of a subgroup $G$ of $\PSp(6, q)$ isomorphic to $\PSL(2, q^2)$. The group $G$ leaves invariant a parabolic quadric $\cQ(4, q)$ with the property that exactly $q^2+1$ of its lines are lines of $\cW(5, q)$. In Section~\ref{iso}, we show that the obtained $m$-ovoids do not fall in any of the examples known so far in the literature. Finally, in Section~\ref{W(5,2)} we prove with the aid of a computer that $\cW(5, 2)$ has $m$-ovoids if and only if $m = 3$ and that there are exactly three pairwise non-isomorphic examples. The first example arises from an elliptic quadric $\cQ^-(5, 2)$ polarizing to $\cW(5, 2)$, whereas the other two are the $3$-ovoids described in Theorem~\ref{first} and Theorem~\ref{second}. 

\section{Constructions}

Throughout the paper we assume $q$ to be an even prime power. 

\subsection{$\left(\frac{q^n-1}{q-1}\right)$-ovoids of $\cW(2n+1, q)$, $n \ge 2$, from elliptic quadrics}\label{FirstConstr}

Let $\PG(2n+1, q)$, $n \ge 2 $, be the $(2n+1)$-dimensional projective space over $\F_{q}$ equipped with homogeneous projective coordinates $(X_1, X_2, \dots, X_{2n+2})$. Let $\delta \in \F_q$ be such that the polynomial $X^2 + X + \delta$ is irreducible over $\F_q$ and let us consider the following quadrics of $\PG(2n+1, q)$
\begin{align*}
    & \cQ_{\mu}: X_1^2+X_1X_{2n+2}+\delta X_{2n+2}^2+ \sum_{i = 2}^{n+1} X_i X_{2n+3-i} + \mu \left(X_{2n}^2+X_{2n}X_{2n+1}+\delta X_{2n+1}^2\right) =0, \mu \in \F_q, \\
    & \cQ_{\infty}:  X_{2n}^2+X_{2n}X_{2n+1}+\delta X_{2n+1}^2  =0. 
\end{align*}
Then $\{\cQ_\mu \colon \mu \in \F_q\} \cup \{\cQ_\infty\}$ is a pencil of quadrics, where $\cQ_\mu$ is an elliptic quadric for every $\mu \in \F_q$ (cf. \cite[Theorem 1.2]{HT}) and $\cQ_\infty$ is a degenerate quadric consisting of the points of the $(2n-1)$-space 
\begin{align*}
\Delta: X_{2n} = X_{2n+1} = 0.    
\end{align*} 
The base locus of the pencil consists of the cone of $\Delta$ having as vertex the line 
\begin{align*}
\ell: X_1 =0, X_4 = \dots = X_{2n+1} = 0,    
\end{align*} 
and as base an elliptic quadric $\cQ^-(2n-3, q)$. The symplectic polarity $\perp_\mu$ associated with $\cQ_\mu$ is given by the bilinear form with Gram matrix
\begin{align*}
& J_\mu = \begin{pmatrix}
0 & 0 & 0 & 0 & \hdots & 0 & 0 & 0 & 1 \\
0 & 0 & 0 & 0 & \hdots & 0 & 0 & 1 & 0 \\
0 & 0 & 0 & 0 & \hdots & 0 & 1 & 0 & 0 \\
0 & 0 & 0 & 0 & \hdots & 1 & 0 & 0 & 0 \\
\vdots & \vdots & \vdots & \vdots & \ddots & \vdots &\vdots & \vdots & \vdots \\ 
0 & 0 & 0 & 1 & \hdots & 0 & 0 & 0 & 0 \\
0 & 0 & 1 & 0 & \hdots & 0 & 0 & \mu & 0 \\
0 & 1 & 0 & 0 & \hdots & 0 & \mu & 0 & 0 \\
1 & 0 & 0 & 0 & \hdots & 0 & 0 & 0 & 0 \\
\end{pmatrix}.
\end{align*}

Let $\cW_\mu$ be the symplectic polar space of $\PG(2n+1, q)$ with polarity $\perp_\mu$. Then 
\[
\ell^{\perp_\mu}=\Delta.
\]
Our goal will be to show that $\cQ_\mu$, $\mu \in \F_q \setminus \{0\}$ is a $\left(\frac{q^n-1}{q-1}\right)$-ovoid of $\cW_0$, despite the fact that it is an elliptic quadric not polarizing to $\cW_0$.

Let $\chi$ be a $(2n-s)$-space of $\PG(2n+1, q)$ and $\Upsilon$ an $s$-space of $\chi$. Recall that a projectivity of $\PG(2n+1,q)$ fixing pointwise $\chi$ and stabilizing the $(s+1)$-spaces through $\Upsilon$ is called {\em elation with axis $\chi$ and center $\Upsilon$}. 

\begin{lemma}\label{commute}
If $\mu \in \F_q \setminus \{0\}$, then the two polarities $\perp_0$ and $\perp_\mu$ commute. Moreover $\perp_{0} \perp_{\mu} = \perp_{\mu} \perp_0$ is the involutory elation $\eta$ of $\PG(2n+1, q)$ having axis $\Delta$ and center $\ell$.
\end{lemma}
\begin{proof}
It is enough to note that
\begin{align*}
& (J_0)^{-t} J_\mu = (J_\mu)^{-t} J_0 = 
\begin{pmatrix}
1 & 0 & 0 & 0 & \hdots & 0 & 0 & 0 & 0 \\
0 & 1 & 0 & 0 & \hdots & 0 & \mu & 0 & 0 \\
0 & 0 & 1 & 0 & \hdots & 0 & 0 & \mu & 0 \\
0 & 0 & 0 & 1 & \hdots & 0 & 0 & 0 & 0 & \\
\vdots & \vdots & \vdots & \vdots & \ddots & \vdots &\vdots & \vdots & \vdots \\ 
0 & 0 & 0 & 0 & \hdots & 1 & 0 & 0 & 0 \\
0 & 0 & 0 & 0 & \hdots & 0 & 1 & 0 & 0 \\
0 & 0 & 0 & 0 & \hdots & 0 & 0 & 1 & 0 \\
0 & 0 & 0 & 0 & \hdots & 0 & 0 & 0 & 1 \\
\end{pmatrix}.
\end{align*}
\end{proof}

\begin{lemma}\label{lemma:perp}
Let $\mu \in \F_q$. If $\Xi$ is a $t$-space of $\Delta$, then $\Xi^{\perp_0} = \Xi^{\perp_\mu}$.
\end{lemma}
\begin{proof}
The claim follows since $\left(\Xi^{\perp_0}\right)^{\perp_\mu} = \left(\Xi^{\perp_\mu}\right)^{\perp_0} = \Xi^\eta = \Xi$ by Lemma~\ref{commute}.
\end{proof}

\begin{lemma}\label{lemma:perp1}
Let $\mu \in \F_q$. A generator $g$ of $\cW_\mu$ has at least one point in common with $\ell$ if and only if $g \cap \Delta$ is at least an $(n-1)$-space. 
\end{lemma}
\begin{proof}
Let $g$ be a generator of $\cW_\mu$. The line $\ell$ is contained in $g$ if and only if $g \subset \ell^{\perp_\mu}$, with $\ell^{\perp_\mu} = \Delta$, whereas $g \cap \ell$ is a point if and only if $\langle g, \ell^{\perp_\mu} \rangle = \langle g, \Delta \rangle$ is the hyperplane $(g \cap \ell)^{\perp_\mu}$ and hence $g \cap \Delta$ is an $(n-1)$-space.
\end{proof}

\begin{lemma}\label{lemma:perp2}
Let $\mu \in \F_q \setminus \{0\}$. If $g$ is a generator of $\cW_0$, then $g$ is a generator of $\cW_\mu$ if and only if $g \cap \Delta$ is at least an $(n-1)$-space.
\end{lemma}
\begin{proof}
Let $g$ be a generator of $\cW_0$, i.e. $g = g^{\perp_0}$. If $g \subset \Delta$, then $g = g^{\perp_0} = g^{\perp_\mu}$ by Lemma~\ref{lemma:perp}. If $g \cap \Delta$ is an $(n-1)$-space, then $(g \cap \Delta)^{\perp_\mu} = (g \cap \Delta)^{\perp_0}$ by Lemma~\ref{lemma:perp}. Hence the $q+1$ generators of $\cW_\mu$ through $g \cap \Delta$ are the same as those of $\cW_0$. Since $g$ passes through $g \cap \Delta$, we have that $g = g^{\perp_\mu}$. 

Vice versa if $g$ is a generator of $\cW_0$ such that $g \cap \Delta$ is an $(n-2)$-space, then $|g \cap \ell| = 0$ by Lemma~\ref{lemma:perp1} and hence $g^{\perp_\mu} = g^\eta \ne g$.
\end{proof}

\begin{lemma}\label{lemma:perp3}
Let $\mu \in \F_q \setminus \{0\}$. If $g$ is a generator of $\cW_0$ and $g \cap \Delta$ is an $(n-2)$-space, then $g \cap \cQ_\mu$ is a cone having as vertex an $(n-3)$-space and as base a non-degenerate conic.
\end{lemma}
\begin{proof}
Let $g$ be a generator of $\cW_0$. If $g \cap \Delta$ is an $(n-2)$-space, then $|g \cap \ell| = 0$ by Lemma~\ref{lemma:perp1}. We claim that $g \cap \Delta \cap \cQ_\mu$ is an $(n-3)$-space. Let $\Delta'$ be a $(2n-3)$-space of $\Delta$ containing $g \cap \Delta$ and disjoint from $\ell$. Recall that $\Delta \cap \cQ_{\mu}$ is a cone having as vertex the line $\ell$ and as base a $\cQ^-(2n-3, q) = \cQ_\mu \cap \Delta'$. Let $\cW(2n-3, q)$ be the symplectic polar space of $\Delta'$ associated with $\cQ^-(2n-3, q)$. In other words $\cW(2n-3, q)$ is the symplectic polar space induced by each of the $\cW_\mu$ on $\Delta'$. Since $g \cap \Delta \subseteq (g \cap \Delta)^{\perp_0} = (g \cap \Delta)^{\perp_{\mu}}$ and $g \cap \Delta$ is an $(n-2)$-space, it follows that $g \cap \Delta$ is a generator of $\cW(2n-3, q)$. On the other hand, every generator of $\cW(2n-3, q)$ meets $\cQ^-(2n-3, q)$ in an $(n-3)$-space. Hence $g \cap \Delta \cap \cQ_\mu = g \cap \Delta \cap \cQ^-(2n-3, q)$ is an $(n-3)$-space, as claimed. 

Note that $(g \cap \Delta)^{\perp_0} = (g \cap \Delta)^{\perp_\mu}$ is an $(n+2)$-space. Moreover, $g \cap \Delta \simeq \PG(n-2, q)$ and $g \cap \Delta \cap \cQ_\mu \simeq \PG(n-3, q)$ imply that $(g \cap \Delta)^{\perp_0} \cap \cQ_\mu = (g \cap \Delta)^{\perp_\mu} \cap \cQ_\mu$ is a cone having as vertex $g \cap \Delta \cap \cQ_\mu$ and as base a parabolic quadric $\cQ(4, q)$, see \cite[Theorem 1.51]{HT}. Let $\Lambda$ be a $4$-space contained in $(g \cap \Delta)^{\perp_0}$ and disjoint from $g \cap \Delta \cap \cQ_\mu$. Then, necessarily, by dimensional considerations, $\Lambda \cap g \cap \Delta$ is a point and $\Lambda \cap g$ is a plane. In particular, $\Lambda \cap \cQ_{\mu}$ is a $\cQ(4, q)$ having as nucleus the point $\Lambda \cap g \cap \Delta$. Of course $\Lambda \cap g \cap \Delta \subset \Lambda \cap g$. Hence $|\Lambda \cap g \cap \cQ_\mu| = q+1$. It follows that $g$ intersects $\cQ_\mu$ in a cone having as vertex the $(n-3)$-space $g \cap \Delta \cap \cQ_\mu$ and as base the $q+1$ points of $\Lambda \cap g \cap \cQ_\mu$. Observe that $\Lambda \cap g \cap \cQ_\mu$ is not a line, otherwise $g \cap \cQ_\mu$ would be an $(n-1)$-space and $g$ a generator of $\cW_\mu$, contradicting Lemma~\ref{lemma:perp2}. Therefore $\Lambda \cap g \cap \cQ_\mu$ is a non-degenerate conic, as stated.
\end{proof}

\begin{theorem}\label{first}
$\cQ_{\mu}$, $\mu \in \F_q \setminus \{0\}$ is a $\left(\frac{q^n-1}{q-1}\right)$-ovoid of $\cW_0$.
\end{theorem}
\begin{proof}
Let $g$ be a generator of $\cW_0$. If $g \cap \Delta$ is at least an $(n-1)$-space, then $g$ is a generator of $\cW_\mu$ by Lemma~\ref{lemma:perp2}. Hence $g \cap \cQ_\mu$ is a generator of $\cQ_\mu$, i.e., an $(n-1)$-space. If $g \cap \Delta$ is an $(n-2)$-space, then by Lemma~\ref{lemma:perp3} $g$ intersects $\cQ_\mu$ in a cone having as vertex the $(n-3)$-space $g \cap \Delta \cap \cQ_\mu$ and as base the $q+1$ points of $\Lambda \cap g \cap \cQ_\mu$. This gives $|g \cap \cQ_\mu| = (q+1) q^{n-2} + \frac{q^{n-2} - 1}{q-1} = \frac{q^n-1}{q-1}$. 
\end{proof}

The next results will be needed in Section \ref{iso}, where the isomorphism issue will be discussed.

\begin{prop}\label{prop:no_disjoint}
Let $\mu \in \F_q \setminus \{0\}$. If $r$ is a totally isotropic line with respect to $\perp_0$, then $r$ is totally isotropic with respect to $\perp_\mu$ if and only if $\vert r \cap \Delta \vert \geq 1$.
\end{prop}
\begin{proof}
Let $r$ be a totally isotropic line w.r.t. $\perp_0$. If $|r \cap \Delta| \ge 1$, let $P \in r \cap \Delta$; by Lemma~\ref{lemma:perp}, $P^{\perp_0}=P^{\perp_\mu}$  and so $P \in r \subset P^{\perp_\mu}$, making $r$ totally isotropic w.r.t. $\perp_\mu$. Vice versa if $r$ is totally isotropic with respect to $\perp_\mu$, we want to prove that it cannot be disjoint from $\Delta$. Assume by contradiction that $|r \cap \Delta| = 0$ and let $g$ be a generator of $\cW_0$ contained in $\Delta$. By Lemma~\ref{lemma:perp2}, $g$ is also a generator of $\cW_\mu$. The $(n-2)$-space $r^{\perp_0} \cap g$ is contained in $\Delta$ and therefore $\left(r^{\perp_0} \cap g\right)^{\perp_0}=\left(r^{\perp_0} \cap g\right)^{\perp_\mu}$. Hence  $r\subseteq \left(r^{\perp_0} \cap g\right)^{\perp_0}=\left(r^{\perp_0} \cap g\right)^{\perp_\mu}$. It follows that $\langle r, r^{\perp_0}\cap g \rangle$ is a generator of both $\cW_0$ and $\cW_\mu$. Moreover such a generator intersects $\Delta$ in an $(n-2)$-space, contradicting Lemma \ref{lemma:perp2}.
\end{proof}

\begin{cor}\label{cor:lines}
Let $\mu \in \F_q \setminus \{0\}$. A line of $\cW_0$ has $0$, $1$, $2$ or $q+1$ points in common with $\cQ_\mu$ and each case occurs.
\end{cor}
\begin{proof}
By Proposition~\ref{prop:no_disjoint}, a line $r$ of $\cW_0$ such that $\vert r \cap \Delta \vert \geq 1$ is a line of $\cW_\mu$ and hence has either $1$ or $q+1$ points in common with $\cQ_\mu$. On the other hand, if $r$ is disjoint from $\Delta$, then it is not a line of $\cW_\mu$ and therefore $r$ meets $\cQ_\mu$ in either $0$ or $2$ points.
\end{proof}

\begin{prop}\label{generators}
Let $\mu \in \F_q \setminus \{0\}$. There are exactly $(q^{n+1} + q^n + 1) \prod_{i = 1}^{n-2} (q^{n-i} + 1)$ generators of $\cQ_\mu$ that are totally isotropic with respect to $\perp_0$. \footnote{Here the product with indices from one to a non positive integer reads as one.}
\end{prop}
\begin{proof}
If $\bar{g}$ is a generator of $\cQ_\mu$ that is totally isotropic with respect to $\perp_0$, then there is a generator $g$ of $\cW_0$ such that $\bar{g} \subset g$. Since $\bar{g}$ is an $(n-1)$-space, $g$ does not intersect $\Delta$ in exactly an $(n-2)$-space by Lemma~\ref{lemma:perp3}. Hence $g \cap \Delta$ is at least an $(n-1)$-space and $g$ is a generator of $\cW_\mu$ by Lemma~\ref{lemma:perp2}. Therefore $g \cap \cQ_\mu = \bar{g}$. It follows that either $\bar{g}$ is contained in $\Delta$ or it meets $\Delta$ in an $(n-2)$-space. In the former case $\bar{g}$ contains $\ell$ and we find 
\begin{align}
& \# (n-3)\mbox{-spaces of } \cQ^-(2n-3, q) = \prod_{i = 1}^{n-2}(q^{n-i} + 1) \label{genDelta}
\end{align}
$(n-1)$-spaces of $\Delta \cap \cQ_\mu$. Hence there are $\prod_{i = 1}^{n-2}(q^{n-i} + 1)$ generators of $\cQ_\mu$ contained in $\Delta$. In the latter case $\bar{g} \cap \ell$ is a point and $\langle \bar{g} \cap \Delta, \ell \rangle$ is an $(n-1)$-space of $\Delta \cap \cQ_\mu$. On the other hand, an $(n-1)$-space of $\Delta \cap \cQ_\mu$ (and that hence contains $\ell$) has precisely $q^{n-1} + q^{n-2}$ $(n-2)$-spaces of $\Delta \cap \cQ_\mu$ meeting the line $\ell$ in one point. Therefore, taking into account \eqref{genDelta}, there are precisely 
\[
\left( q^{n-1} + q^{n-2} \right) \times \# (n-3)\mbox{-spaces of } \cQ^-(2n-3, q)
\]
$(n-2)$-spaces of $\Delta \cap \cQ_\mu$ having exactly one point in common with $\ell$. Since through an $(n-2)$-space of $\Delta \cap \cQ_\mu$ intersecting $\ell$ in one point, there pass $q^2+1$ generators of $\cQ_\mu$ and exactly one of these is contained in $\Delta$, it turns out that there are 
\begin{align*}
& q^2 \times\left( q^{n-1} + q^{n-2} \right) \prod_{i = 1}^{n-2}(q^{n-i} + 1) = \left( q^{n+1} + q^n \right) \prod_{i = 1}^{n-2}(q^{n-i} + 1)
\end{align*}
generators of $\cQ_\mu$ meeting $\Delta$ in an $(n-2)$-space.   
\end{proof}

\begin{cor}\label{cor:disjoint}
Let $\mu \in \F_q \setminus \{0\}$. There are at most $q+1$ pairwise disjoint generators of $\cQ_\mu$ that are totally isotropic with respect to $\perp_0$.
\end{cor}
\begin{proof}
By the proof of Proposition~\ref{generators}, a generator of $\cQ_\mu$ which is totally isotropic with respect to $\perp_0$ has at least one point in common with $\ell$. Hence there are at most $q+1$ of these $(n-1)$-spaces.
\end{proof}

\subsection{A class of $(q+1)$-ovoids of $\cW(5, q)$ admitting $\PSL(2, q^2)$ as an automorphism group}\label{SecondConstr}

In this subsection we provide a class of $(q+1)$-ovoids of $\cW(5, q)$. In order to do that a preliminary description of the geometric setting is required. Let $\PG(5, q^2)$ be the $5$-dimensional projective space over $\F_{q^2}$ equipped with homogeneous projective coordinates $(X_1, X_2, X_3, X_4, X_5, X_6)$. The set of vectors
\begin{align*}
& \{(\alpha, \alpha^q, \delta_0, \beta, \beta^q, \delta_1) \colon \alpha, \beta \in \F_{q^2}, \delta_0, \delta_1 \in \F_q \}
\end{align*}
is an $\F_q$-vector space and hence the corresponding projective points are those of a Baer subgeometry $\Sigma$ of $\PG(5, q^2)$. Denote by $\tau$ the involutory collineation of $\PG(5, q^2)$ fixing pointwise $\Sigma$, that is
\begin{align*}
& \tau: X_1' = X_2^q, X_2' = X_1^q, X_3' = X_3, X_4' = X_5^q, X_5' = X_4^q, X_6' = X_6.
\end{align*}

Let $\bar{\Pi}$ be the hyperplane of $\PG(5, q^2)$ with equation $X_6 = 0$ and let $\bar{\cQ}$ be the parabolic quadric of $\bar{\Pi}$ given by 
\begin{align*}
X_3^2 + X_1 X_5 + X_2 X_4 = 0. 
\end{align*}
Thus $\Pi = \bar{\Pi} \cap \Sigma$ and $\cQ = \bar{\cQ} \cap \Sigma$ are a hyperplane and a parabolic quadric of $\Sigma$, respectively. In particular 
\begin{align*}
\cQ = \{(\alpha, \alpha^{q}, \sqrt{\alpha \beta^q + \alpha^q \beta}, \beta, \beta^q, 0) \colon \alpha, \beta \in \F_{q^2}, (\alpha, \beta) \ne (0,0)\}
\end{align*} 
and its nucleus is the point $N = (0,0,1,0,0,0)$. Fix an element $\w \in \F_{q^2} \setminus \F_q$ such that $\w + \w^q = 1$ and let $\cH(5, q^2)$ be the Hermitian variety of $\PG(5, q^2)$ defined by 
\begin{align*}
 \w X_1 X_4^q + \w^q X_1^q X_4 + \w^q X_1 X_6^q+ \w X_1^q X_6 + \w^q X_2 X_5^q + \w X_2^q X_5 + & \\
 \w X_2 X_6^q + \w^q X_2^q X_6 + X_3 X_6^q + X_3^q X_6 & = 0. 
\end{align*}
Since $\Sigma \subset \cH(5, q^2)$, the polarity induced by $\cH(5, q^2)$ on $\Sigma$ is symplectic \cite[Lemma 6.2]{CMPS}. Denote by $\perp$ and $\cW(5, q)$ the symplectic polarity and polar space induced by $\cH(5, q^2)$ on $\Sigma$. Thus $N^\perp = \Pi$. Let $G$ be the subgroup of $\PGL(6, q^2)$ generated by the matrices
\begin{align*}
& M_{a,b,c,d} = 
\begin{pmatrix}
a^2 & 0 & 0 & 0 & c^2 & \frac{c (a + c \w^q)}{\w} \\
0 & a^{2q} & 0 & c^{2q} & 0 & \frac{c^q (a^q + c^q \w)}{\w^q} \\ 
ab & a^q b^q & 1 & c^q d^q & cd & \frac{d (a + c \w^q)}{\w} + \frac{d^q (a^q + c^q \w)}{\w^q} + \frac{1}{\w^{q+1}} \\
0 & b^{2q} & 0 & d^{2q} & 0 & \frac{d^q (b^q + d^q \w) + \w}{\w^q} \\
b^2 & 0 & 0 & 0 & d^2 & \frac{d (b + d \w^q) + \w^q}{\w} \\
0 & 0 & 0 & 0 & 0 & 1 \\
\end{pmatrix}, & a, b, c, d \in \F_{q^2}, ad+bc = 1.
\end{align*} 
The map $\begin{pmatrix} a & b \\ c & d \end{pmatrix} \mapsto M_{a,b,c,d}$ is an isomorphism between  $G$ and $\PSL(2, q^2)$. Let 
\begin{align*}
& J = \begin{pmatrix}
0 & 0 & 0 & \w & 0 & \w^q \\
0 & 0 & 0 & 0 & \w^q & \w \\
0 & 0 & 0 & 0 & 0 & 1 \\
\w^q & 0 & 0 & 0 & 0 & 0 \\
0 & \w & 0 & 0 & 0 & 0 \\
\w & \w^q & 1 & 0 & 0 & 0 \\
\end{pmatrix}
\end{align*}
be the Gram matrix of the sesquilinear form defining $\cH(5, q^2)$. The group $G$ fixes the Hermitian variety $\cH(5, q^2)$, since $M_{a,b,c,d}^t J M_{a,b,c,d}^q = J$, for all $a,b,c,d \in \F_{q^2}$ with $ad+bc = 1$.  Similarly, some straightforward calculations show that $G$ stabilizes the parabolic quadric $\bar{\cQ}$ (see \cite[Lemma 1.5.23]{CD}). 

Let $\gamma \in \F_{q^2}$ be such that $X^2 + X + \gamma = 0$ is irreducible over $\F_{q^2}$ and consider the point 
\begin{align*}
S_\gamma = \left( \frac{\gamma}{\w}, \frac{\gamma^q}{\w^q}, \frac{\w^q \gamma}{\w} + \frac{\w \gamma^q}{\w^q}, 1, 1, 1\right) \in \Sigma. 
\end{align*}
Set $\cO = \cQ \cup S_\gamma^G$. We will prove that $\cO$ is a $(q+1)$-ovoid of $\cW(5, q)$.
 
The hyperplane 
\begin{align*}
S_\gamma^\perp: X_1 + X_2 + X_3 + \gamma^q X_4 + \gamma X_5 = 0
\end{align*}
of $\Sigma$ meets $\Pi$ in a solid.

\begin{lemma}\label{perp}
$\Pi \cap S_\gamma^\perp \cap \cQ$ is an elliptic quadric and $|S_\gamma^G| = q^2(q^2-1)$.  
\end{lemma}
\begin{proof}
Let $\bar{S_{\gamma}^\perp}$ be the hyperplane of $\PG(5, q^2)$ obtained by extending $S_{\gamma}^\perp$ over $\F_{q^2}$. Then $\pi \cap \bar{S_{\gamma}^\perp}$ is a line $r$ external to $\pi \cap \bar{\cQ}$ and $N \notin r$. 

Note that $Stab_G(S_\gamma) \le Stab_G(\Pi \cap S_{\gamma}^\perp)$. The subgroup of $G$ fixing $\Pi \cap S_{\gamma}^\perp$ is the dihedral group of order $2(q^2+1)$ generated by $M_{a, b, \gamma b, a+b}$, with $a,b,c,d \in \F_{q^2}$, $a^2+ ab + \gamma b^2 = 1$ and $M_{1, 0, 1, 1}$. Some straightforward calculations show that the cyclic group of order $q^2+1$ fixes $S_\gamma$, whereas the involution generated by $M_{1, 0, 1, 1}$ does not. Hence $|Stab_G(S_\gamma)| = q^2+1$ and $|S_\gamma^G| = q^2(q^2-1)$. 
\end{proof}

By Lemma~\ref{perp}, the set $\cO$ has size $(q+1)(q^3+1)$. An analytic description of $S_\gamma^G$ is provided below.
\begin{align*}
& S_\gamma^G = \{P_{x, y} \colon x, y \in \F_{q^2}, x \ne 0 \},
\end{align*} 
where 
\begin{align*}
& P_{x,y} = \left( \frac{x^2}{\w}, \frac{x^{2q}}{\w^q}, \frac{\w^2 \gamma^q + \w^{2q} \gamma + x y \w^q + x^q y^q \w + 1}{\w^{q+1}}, f(x, y)^q, f(x, y), 1 \right) 
\end{align*}
and
\begin{align*}
& f(x, y) = \frac{x^{2} y^{2} + x y + \w^q x^{2} + \gamma}{\w x^2}.
\end{align*}
Observe that $P_{x, y} = P_{x', y'}$, implies $(x, y) = (x', y')$. In order to prove that $\cO$ is a $(q+1)$-ovoid of $\cW(5, q)$, a preliminary study regarding the action of the group $G$ on the lines of $\cW(5, q)$ contained in $\Pi$ and on the generators of $\cW(5, q)$ is needed. These properties are acquired in various results. In particular, in Lemma~\ref{lines}, it is shown that $G$ has four orbits on lines of $\cW(5, q)$ contained in $\Pi$, according as they have $0$, $1$, $2$ or $q+1$ points in common with $\cQ$, whereas in Lemma~\ref{plane}, Lemma~\ref{disjoint} and Proposition~\ref{conic}, it is proven that the generators of $\cW(5, q)$, that are either contained in $\Pi$ or intersect $\Pi$ in a line having one or $q+1$ points in common with $\cQ$, meet $\cO$ in $q+1$ points. 

Observe that both the point $N$ and the hyperplane $\bar{\Pi}$ are left invariant by $G$. In particular, $G$ fixes the following planes of $\bar{\Pi}$:
\begin{align*}
& \pi: X_2 = X_4 = X_6 = 0 \mbox{ and } \pi^\tau: X_1 = X_5 = X_6 = 0,
\end{align*}
where $\pi \cap \Sigma = \pi^\tau \cap \Sigma = \{N\}$. The parabolic quadric $\bar{\cQ}$ intersects both $\pi$ and $\pi^\tau$ in a non-degenerate conic and $G$ acts faithfully on both $\pi$ and $\pi^\tau$. Moreover, it is easily seen that $G$ leaves $\Sigma$ invariant. 

If $P \in \pi \cap \bar{\cQ}$, then $\langle P, P^\tau \rangle_{q^2}$ is a line of both $\bar{\cQ}$ and $\cH(5, q^2)$. Since $\Sigma \cap \langle P, P^\tau \rangle_{q^2}$ is a line of $\cQ$, it follows that 
\begin{align*}
\cS = \{\Sigma \cap \langle P, P^\tau \rangle_{q^2} \colon P \in \pi \cap \bar{\cQ}\}
\end{align*}
is a classical line-spread of $\cQ$. Since the Slines of $\cS$ are in bijective correspondence with the points of the conic $\bar{\Pi} \cap \bar{\cQ}$ and $G$ acts in its natural representation on the $q^2+1$ points of the conic $\bar{\Pi} \cap \bar{\cQ}$, it follows that $G$ acts in its natural representation on the $q^2+1$ lines of $\cS$. On the other hand, if $R \in \pi \setminus (\bar{\cQ} \cup \{N\})$, then $\langle R, R^\tau \rangle_{q^2}$ is a line of $\cH(5, q^2)$ and $\Sigma \cap \langle R, R^\tau \rangle_{q^2}$ is a line of $\Sigma$. Hence $\Sigma \cap \langle R, R^\tau \rangle_{q^2}$ is a line of $\cW(5, q)$. 

Let $\sigma_1$ be the plane $\Sigma \cap  \langle R, R^\tau, N \rangle_{q^2}$ of $\Sigma$. Since the line $\langle R, N \rangle_{q^2}$ meets the conic $\pi \cap \bar{\cQ}$ in exactly one point, say $P$, we have that $\sigma_1 \cap \cQ$ is the line $\Sigma \cap \langle P, P^\tau \rangle_{q^2}$ of $\cS$. Therefore $\Sigma \cap \langle R, R^\tau \rangle_{q^2}$ is a line that is tangent to $\cQ$ and each of the $q^2-1$ lines of $\sigma_1$ tangent to $\cQ$ not through $N$ arises from a point of $\langle R, N \rangle_{q^2} \setminus \{P, N\}$ in this way. Moreover $\sigma_1$ is a plane of $\cW(5, q)$.

Let $r$ be a line of $\pi$ not incident with $N$. Since $\Sigma \cap \langle T, T^\tau \rangle_{q^2}$ is a line of $\Sigma$ for every point $T \in r$ and $r$, $r^\tau$ are skew, it follows that 
\begin{align*}
\cD_{r} = \{\Sigma \cap \langle T, T^\tau \rangle_{q^2} \colon T \in r\}
\end{align*}
is a Desarguesian line-spread of a solid $\Gamma_r$ of $\Sigma$. The line $r$ has none or two points of $\pi \cap \bar{\cQ}$; in the former case $\cD_r$ consists of lines of $\cW(5, q)$ that are tangent to $\cQ$, whereas in the latter case $\cD_r$ has two lines of $\cS$ and $q^2-1$ lines of $\cW(5, q)$ that are tangent to $\cQ$. Therefore $|\Gamma_r \cap \cQ|$ equals $q^2+1$ or $(q+1)^2$ and hence $\Gamma_r \cap \cQ$ is an elliptic quadric or a hyperbolic quadric, according as $r$ is external or secant to $\pi \cap \bar{\cQ}$. Moreover, since there are $q^4$ solids of $\Pi$ meeting $\cQ$ in an elliptic quadric or a hyperbolic quadric, it follows that every such solid arises from a line $r$ of $\pi$ not containing $N$. 

Denote by $\sigma_2$ the plane of $\cW(5, q)$ given by $\sqrt{\w} X_1 + \sqrt{\w^q} X_2 = \sqrt{\w^q} X_4 + \sqrt{\w} X_5 = X_6 = 0$. Thus $\sigma_2 \cap \cQ$ consists of the non-degenerate conic 
\begin{align*}
& \cC = \left\{\left( \sqrt{\w^q} \alpha, \sqrt{\w} \alpha, \sqrt{\alpha \beta}, \sqrt{\w} \beta, \sqrt{\w^q} \beta, 0 \right) \colon \alpha, \beta \in \F_{q}, (\alpha, \beta) \ne (0,0) \right\}
\end{align*} 
whose nucleus is the point $N$.

\begin{lemma}\label{plane}
The group $G$ has two orbits $\cG_1$ and $\cG_2$ on planes of $\cW(5, q)$ through $N$. A plane of $\cW(5, q)$ belongs to $\cG_1$ or $\cG_2$ according as it meets $\cQ$ in a line of $\cS$ or in a non-degenerate conic. 
\end{lemma}
\begin{proof}
Let $\cG_1 = \sigma_1^G$. Since $\sigma_1 \cap \cQ$ is a line of $\cS$ and the group $G$ is transitive on the $q^2+1$ lines of $\cS$ we have that $|\cG_1| = q^2+1$. 

Let $\cG_2 = \sigma_2^G$. A member of $G$ fixes $\sigma_2$ if and only if it is generated by $M_{a,b,c,d}$, where $a,b,c,d \in \F_q$, with $ad+bc = 1$. Hence $K = Stab_{G}(\sigma_2) \simeq \PSL(2, q)$ and $|\sigma_2^G| = |G|/(q^3-q) = q^3+q$. 

Note that there are $|\cG_1|+|\cG_2| = (q+1)(q^2+1)$ planes of $\cW(5, q)$ through $N$. This concludes the proof.
\end{proof}

\begin{remark}\label{remark:faithful}
The group $K = Stab_G(\sigma_2) \simeq \PSL(2, q)$ has  a faithful action on $\sigma_2$.
\end{remark}

In view of Remark \ref{remark:faithful}, the group $K$ has three orbits on points of $\sigma_2$, namely the conic $\cC$, its nucleus $N$ and $\sigma_2 \setminus (\cC \cup \{N\})$. Since every point of $\Pi \setminus (\cQ \cup \{N\})$ lies on $q^2$ planes of $\cG_2$, the next result follows.

\begin{lemma}\label{points}
The group $G$ has three orbits on points of $\Pi$, namely $\{N\}$, $\cQ$ and their complement. 
\end{lemma}

\begin{lemma}\label{lines}
The group $G$ has four orbits on lines of $\cW(5, q)$ in $\Pi$ not through $N$. They are $\cS$ and $\cL_i$, $i = 0, 1, 2$, and a line of $\cL_i$ has $i$ points in common with $\cQ$.
\end{lemma}
\begin{proof}
From the discussion above, if $\ell$ is a line of $\sigma_1$ tangent to $\cQ$ and not incident with $N$, then there exists a point $R \in \pi \setminus (\bar{\cQ} \cup \{N\})$ such that $\ell = \Sigma \cap \langle R, R^\tau \rangle_{q^2}$. Since the group $G$ is transitive on the $q^2-1$ points of $\langle R, N \rangle_{q^2}$ not on $\bar{\cQ}$ and distinct from $N$, it follows that $G$ permutes in a single orbit the $q^2-1$ lines of $\sigma_1$ tangent to $\cQ$ not through $N$. Let $\cL_1 = \ell^G$. Thus $|\cL_1^G| = (q^2-1)(q^2+1) = q^4-1$.

By Remark \ref{remark:faithful} the group $K$ has three orbits on lines of $\sigma_2$. They have size $q+1$, $q(q+1)/2$, $q(q-1)/2$ according as they contain tangent, secant or external lines to $\cC$. Two distinct planes of $\cG_2$ meet either in $\{N\}$ or in a line of $\cW(5, q)$ through $N$ that hence is tangent to $\cQ$. A line of $\sigma_2$ not containing $N$ meets $\cQ$ in $0$ or $2$ points. Let $\ell_i$, $i = 0,2$, be a line of $\sigma_2$ with $i$ points in common with $\cQ$ and let $\cL_i = \ell_i^G$. It follows that $|\cL_0| = (q^3+q)q(q-1)/2 = q^2(q^2+1)(q-1)/2$ and $|\cL_2| = (q^3+q)q(q+1)/2 = q^2(q^2+1)(q+1)/2$.

Note that there are $|\cS|+|\cL_0|+|\cL_1|+|\cL_2| = q^2(q+1)(q^2+1)$ lines of $\cW(5, q)$ in $\Pi$ not through $N$. This concludes the proof.
\end{proof}

\begin{lemma}\label{disjoint}
A plane of $\cW(5, q)$ containing a line of $\cS$ is disjoint from $S_\gamma^G$.
\end{lemma}
\begin{proof}
Let $\sigma'$ be a plane of $\cW(5, q)$ containing a line $\ell$ of $\cS$. If $\sigma' \subset \Pi$ the statement is trivial. If $\sigma' \not\subset \Pi$ and $|\sigma' \cap S_\gamma^G| \ne 0$, then we may assume without loss of generality that $S_\gamma \in \sigma'$. Thus $\Pi \cap S_\gamma^\perp \cap \cQ$ is an elliptic quadric containing the line $\ell$, a contradiction. 
\end{proof}

Let us consider the following representative $\ell_1$ of the line-orbit $\cL_1$.
\begin{align*}
& \ell_1 = \langle (1, 1, 0, 0, 0, 0), (\w, \w^q, 1, 0, 0, 0) \rangle_q = \Sigma \cap \langle (1,0,1,0,0,0), (0,1,1,0,0,0) \rangle_{q^2}.
\end{align*} 
Thus 
\begin{align*}
& \ell_1^\perp : X_4 = X_5 = X_6.
\end{align*} 
Let $\sigma \subset \ell_1^\perp$ be the plane of $\cW(5, q)$ given by 
\begin{align*}
X_4 + X_6 = X_5 + X_6 = X_1 + X_2 + X_3 + (\gamma + \gamma^q) X_6 = 0.
\end{align*}
Hence $\sigma \cap \Pi = \ell_1$. 

\begin{lemma}\label{tangent}
The group $G$ has one orbit, say $\cG$, on planes of $\cW(5, q)$ intersecting $\Pi$ in exactly a line of $\cL_1$.
\end{lemma}
\begin{proof}
There are $q(q^4-1)$ planes of $\cW(5, q)$ intersecting $\Pi$ in exactly a line of $\cL_1$. Let $\cG = \sigma^G$. The stabilizer of $\sigma$ in $G$ is a subgroup of $Stab_G(\ell_1)$, where $Stab_G(\ell_1)$ has order $q^2$ and is generated by $M_{1, 0, c, 1}$, with $c \in \F_{q^2}$. On the other hand, some calculations show that the projectivity generated by $M_{1, 0, c, 1}$ fixes $\sigma$ if and only if $c = \mu \sqrt{\w}$, $\mu \in \F_q$. Therefore $|Stab_G(\sigma)| = q$ and $|\cG| = |\sigma^G| = q(q^4-1)$, as required. 
\end{proof}

\begin{prop}\label{conic}
A plane of $\cW(5, q)$ belonging to $\cG$ meets $\cO$ in a non-degenerate conic.
\end{prop}
\begin{proof}
By Lemma~\ref{tangent} it is enough to show that $\sigma \cap \cO$ is a non-degenerate conic. First of all observe that $\sigma \cap \cQ = \ell_1 \cap \cQ = \{(1, 1, 0,0,0,0)\}$. A point $P_{x, y} \in S_\gamma^G$ belongs to $\sigma$ if and only if 
\begin{align}
& x^2 y^2 + x^2 + x y + \gamma = 0, \label{l1:eq1} \\ 
& \w^q (x^2 + xy + \gamma) + \w (x^{2q} + x^q y^q + \gamma^q) + 1 = 0. \label{l1:eq2}
\end{align}   
Hence equation \eqref{l1:eq1} gives $(xy)^2 = x^2+xy+\gamma$. Since $x \ne 0$, by substituting it in \eqref{l1:eq2} we get 
\begin{align*}
& y = \frac{\sqrt{\w^q} + \mu^2}{\sqrt{\w^q} x},
\end{align*}
where $\mu \in \F_q$. By using again \eqref{l1:eq1} we obtain 
\begin{align}
& x = \frac{\mu^2 + \sqrt[4]{\w^q}\mu + \sqrt{\w^q \gamma}}{\sqrt{\w^q}}, \label{l1:x}
\end{align}
and hence
\begin{align}
& y = \frac{\sqrt{\w^q} + \mu^2}{\mu^2 + \sqrt[4]{\w^q}\mu + \sqrt{\w^q \gamma}}. \label{l1:y}
\end{align}
Note that the right hand side of equation \eqref{l1:x} cannot be zero, since it is irreducible over $\F_{q^2}$ as a polynomial in $\mu$. Taking into account \eqref{l1:x} and \eqref{l1:y}, it follows that the plane $\sigma$ has the $q$ points
\begin{align*}
\left( \frac{\mu^4 + \sqrt{\w^q} \mu^2 + \w^q \gamma}{\w^{q+1}}, \frac{\mu^4 + \sqrt{\w} \mu^2 + \w \gamma^q}{\w^{q+1}}, \frac{\mu^2 + \w^{2q} \gamma + \w^2 \gamma^q}{\w^{q+1}}, 1, 1, 1 \right), \mu \in \F_q, 
\end{align*}
in common with $S_\gamma^G$. Moreover the $q+1$ points of $\sigma \cap \cO$ are those of the non-degenerate conic of $\sigma$ given by 
\begin{align*}
 \w^{2(q+1)} \left(X_1^2 + X_2^2\right) + \w^{q+1} \sqrt{\w} X_1 X_6 + \w^{q+1} \sqrt{\w^q} X_2 X_6 + & \\
 (\w^{2q} \gamma^2 + \w^2 \gamma^{2q} + \w^q \sqrt{w} \gamma + \w \sqrt{\w^q} \gamma^q) X_6^2 & = 0.    
\end{align*}
The proof is now complete.
\end{proof}

\begin{theorem}\label{second}
The set $\cO$ is a $(q+1)$-ovoid of $\cW(5, q)$.
\end{theorem}
\begin{proof}
In order to prove that $\cO$ is a $(q+1)$-ovoid, it is equivalent to show that for a point $V \in \Sigma$, $|V^\perp \cap \cO|$ equals $q(q^2+1) + 1$ or $(q+1)(q^2+1)$, according as $V \in \cO$ or $V \in \Sigma \setminus \cO$. The proof is based on the fact that a generator of $\cW(5, q)$ intersecting $\Pi$ in a line having one or $q+1$ points in common with $\cQ$, meets $\cO$ in $q+1$ points (see Lemma~\ref{disjoint}, Proposition~\ref{conic}). 

Let $V \in \cO$. If $V \in \cQ$, by Lemma~\ref{points}, we may assume that $V = (0,0,0,1,1,0)$. Hence $V^\perp: \w X_1 + \w^q X_2 = 0$ has in common with $S_\gamma^G$ the $q^2(q-1)$ points given by $P_{x, y}$, with $x, y \in \F_{q^2}$, $x+x^q = 0$, $x \neq 0$, whereas $V^\perp \cap \cQ$ is a quadratic cone since $N \in V^\perp$. Therefore $|V^\perp \cap \cO| = q^2(q-1) + q^2+q+1 = q(q^2+1) + 1$. If $V \in S_\gamma^G$, then we may assume that $V = S_\gamma$. By Lemma~\ref{perp} we have that $S_\gamma^\perp \cap \Pi$ is a solid intersecting $\cQ$ in an elliptic quadric. The solid $S_\gamma^\perp \cap \Pi$ possesses a Desarguesian line-spread $\cD_r$ consisting of $q^2+1$ lines of $\cL_1$. The $q^2+1$ planes of $\cW(5, q)$ spanned by $S_\gamma$ and a line of $\cD_r$ are planes of $\cG$ and they cover all the points of $S_\gamma^\perp \setminus \{S_\gamma\}$ exactly once. By Proposition~\ref{conic}, each of these planes has $q+1$ points in common with $\cO$, one of them being $S_\gamma$. Therefore $|V^\perp \cap \cO| = q(q^2+1) + 1$.      

Let $V \in \Sigma \setminus \cO$. If $V = N$, then $V^\perp = \Pi$. Hence $V^\perp \cap \cO = \cQ$ and $|V^\perp \cap \cO| = (q+1)(q^2+1)$. If $V \in \Pi \setminus (\cQ \cup \{N\})$, then by Lemma~\ref{points}, we may assume $V = (0,0,1,1,1,0)$. Hence $V^\perp: \w X_1 + \w^q X_2 + X_6 = 0$ intersects $S_\gamma^G$ in the $q^3$ points given by $P_{x, y}$, with $x, y \in \F_{q^2}$, $x + x^q = 1$; the set $V^\perp \cap \cQ$ is a quadratic cone since $N \in V^\perp$. Therefore $|V^\perp \cap \cO| = q^3 + q^2+q+1 = (q+1)(q^2+1)$. If $V \in \Sigma \setminus (\Pi \cup S_\gamma^G)$, then $V^\perp \cap \Pi$ is a solid not passing through $N$. Hence two possibilities occur: either $V^\perp \cap \Pi$ meets $\cQ$ in an elliptic quadric or in a hyperbolic quadric. Moreover there exists a Desarguesian line-spread $\cD_r$ of $V^\perp \cap \Pi$. If $V^\perp \cap \Pi \cap \cQ$ is an elliptic quadric, then $\cD_r$ consists of lines of $\cL_1$. In this case the $q^2+1$ planes of $\cW(5, q)$ spanned by $V$ and a line of $\cD_r$ are planes of $\cG$. By Proposition~\ref{conic}, each of these planes has $q+1$ points in common with $\cO$. Since $V \notin \cO$, it follows that $|V^\perp \cap \cO| = (q+1)(q^2+1)$. If $V^\perp \cap \Pi \cap \cQ$ is a hyperbolic quadric, then $\cD_r$ has two lines of $\cS$ and $q^2-1$ lines of $\cL_1$. In this case the $q^2-1$ planes of $\cW(5, q)$ spanned by $V$ and a line of $\cD_r$ are planes of $\cG$ and two of them meet $\Pi$ in a line of $\cS$. Since $V \notin \cO$, by Proposition~\ref{conic} and Lemma~\ref{disjoint} we obtain that $|V^\perp \cap \cO| = (q+1)(q^2+1)$. 
\end{proof}

\begin{remark}
The set $\cO$ is not a quadric. More precisely, there exists no quadric $\cA$ of $\Sigma$ such that the set $S_\gamma^G$ coincides with $\cA \setminus \Pi$. Assume by contradiction that $\cA$ is a quadric of $\Sigma$ such that $\cA \setminus \Pi = S_\gamma^G$. First of all observe that $\cA$ would intersect $\Pi$ in $\cQ$. Indeed, by Proposition~\ref{conic}, Lemma~\ref{tangent} and Lemma~\ref{lines}, a plane of $\cG$ intersects $\cA \setminus \Pi$ in $q$ points of a conic and such a conic has its remaining point on $\cQ$. Since through five points of a plane no three on a line there pass a unique conic (\cite[Corollary 7.5]{H1}), then necessarily $\cQ \subset \cA$, whenever $q \ge 8$. Next consider the line $l$ joining the points $R = (0,0,0,1,1,0) \in \cQ$ and $P_{x,{y_{1}}} \in S_\gamma^G$, where $x$ is a fixed element in $\F_{q^2}$ such that $x + x^q = 1$ with $\w^q x^2 + \w x^{2q} \ne 0$ and $y_1$ is an element of $\F_{q^2}$ satisfying the equation $x^2 y_1^2 + x y_1 + x^2 + \gamma = 0$. Some calculations show that $l$ meets $S_\gamma^G$ at the further point $P_{x, y_1 + \frac{\w (x^2 + x^{2q})}{x(\w^q x^2 + \w x^{2q})}} = P_{x, y_1} + \left(\frac{x+x^{q}}{\w^q x^2 + \w x^{2q}}\right)^2 R$. Therefore $|l \cap \cA| = 3$, contradicting the assumption that $\cA$ is a quadric. If $q \in \{2, 4\}$, some computations performed with Magma \cite{magma} confirm that no quadric $\cA$ of $\Sigma$ meets $\Sigma \setminus \Pi$ exactly in $S_\gamma^G$.
\end{remark}

\begin{prop}\label{contained}
A line of $\cW(5, q)$ has $0$, $1$, $2$ or $q+1$ points in common with $\cO$. In particular there are exactly $q^2+1$ lines of $\cW(5, q)$ contained in $\cO$.
\end{prop}
\begin{proof}
Let $\ell$ be a line of $\cW(5, q)$. If $\ell \subset \Pi$ then $\ell \cap \cO = \ell \cap \cQ$ and hence $\ell$ has $0$, $1$, $2$ or $q+1$ points in common with $\cO$. If $\ell \cap \Pi$ is a point and $|\ell \cap S_\gamma^G| \ge 1$, we denote by $P$ a point of $S_\gamma^G \cap \ell$. The solid $P^\perp \cap \Pi$ intersects $\cQ$ in an elliptic quadric and possesses a Desarguesian line-spread $\cD_r$ consisting of $q^2+1$ lines of $\cL_1$. Let $r_1$ be the unique line of $\cD_r$ intersecting $\ell$ in one point. The plane $g$ of $\Sigma$ spanned by $\ell$ and $r_1$ is a generator of $\cW(5, q)$ belonging to $\cG$. By Lemma~\ref{conic}, the plane $g$ meets $\cO$ in a conic, and hence $\ell$ meets $\cO$ in either one or two points.

As a consequence we have that a line of $\cW(5, q)$ contained in $\cO$ has to be contained in $\cQ$. By Lemma~\ref{lines} the lines having this property are exactly those of $\cS$.  
\end{proof}

\section{The isomorphism issue}\label{iso}

An $m$-ovoid of $\cW(2n+1, q)$ obtained by field reduction will be called {\em classical}. A classical $\left(\frac{q^n-1}{q-1}\right)$-ovoid of $\cW(2n+1, q)$ is acquired solely from a $\left( \frac{q^n-1}{q^2-1} \right)$-ovoid of $\cH(n, q^2)$ for an even $n$, that is by applying field reduction to the points of a Hermitian polar space $\cH(n, q^2)$, with $n$ even. In this case, if $q$ is even, the pointset obtained in $\cW(2n+1, q)$ is an elliptic quadric $\cQ^-(2n+1, q)$ of $\PG(2n+1, q)$ polarizing to $\cW(2n+1, q)$. In particular, a generator of $\cW(2n+1, q)$ meets $\cQ^-(2n+1, q)$ in a generator of $\cQ^-(2n+1, q)$, that is, an $(n-1)$-space, whereas a line of $\cW(2n+1, q)$ has either $1$ or $q+1$ points in common with $\cQ^-(2n+1, q)$. 
By subsection~\ref{FirstConstr}, there exists an elliptic quadric of $\PG(2n+1, q)$ not polarizing to $\cW(2n+1, q)$ which gives rise to a $\left(\frac{q^n-1}{q-1}\right)$-ovoid of $\cW(2n+1, q)$. Such an example is not classical due to Corollary~\ref{cor:lines} and by Corollary~\ref{cor:disjoint} it does not come from an $(n-1)$-system of $\cW(2n+1, q)$.  

In the case $n=2$, there are two further known constructions of non-classical $(q+1)$-ovoids of $\cW(5, q)$, $q > 2$, see \cite{CP}. In both cases the geometric setting consists of a pencil $\cP$ of quadrics comprising $q/2$ elliptic quadrics, say $\cE_1, \dots, \cE_{q/2}$, $q/2$ hyperbolic quadrics and a cone with basis a parabolic quadric, say $\cQ(4,q)$, and as vertex its nucleus $V$. Each elliptic quadric of $\cP$ polarizes to the same symplectic polar space $\cW(5, q)$. Let $\perp$ be the polarity of $\PG(5, q)$ defining $\cW(5, q)$. The base locus of $\cP$ is the parabolic quadric $\cQ(4, q)$ and the cone of $\cP$ is contained in $V^\perp$. 

The first construction arises from relative hemisystems of elliptic quadrics. A {\em relative hemisystem} of $\cE_i$ is a set $\cR_i$ of $q^2(q^2-1)/2$ points of $\cE_i \setminus \cQ(4, q)$ such that every line of $\cE_i$ meeting $\cQ(4, q)$ in one point has exactly $q/2$ points in common with $\cR_i$. Consider two distinct elliptic quadrics $\cE_j$, $\cE_k$ of $\cP$, for some fixed indexes $j$, $k$, with $j \ne k$, together with a relative hemisystem $\cR_j \subset \cE_j$, $\cR_k \subset \cE_k$ and take the union $\cO_1 = \cR_j \cup \cR_k \cup \cQ(4, q)$. Then $\cO_1$ is a $(q+1)$-ovoid of $\cW(5, q)$.

The second construction occurs if $q = 2^{2h+1}$, $h > 0$, and comes from the action of the Suzuki group $Sz(q)$ fixing a Suzuki-Tits ovoid $\cT$ of $\cQ(4, q)$. In particular $Sz(q)$ has two orbits, $\cT$ and $\cQ(4, q) \setminus \cT$ on points of $\cQ(4, q)$, and two orbits $A_i$, $B_i$ on points of $\cE_i \setminus \cQ(4, q)$, $i = 1, \dots, q/2$, with sizes $|A_i| = (q^3-q^2)(q-\sqrt{2q}+1)/2$ and $|B_i| = (q^3-q^2)(q+\sqrt{2q}+1)/2$. A line of $\cE_i$ intersecting $\cQ(4, q)$ in a point of $\cT$ has $q/2$ points in common with both $A_i$ and $B_i$, whereas a line of $\cE_i$ intersecting $\cQ(4, q)$ in a point off $\cT$ meets $A_i$ in $(q- \sqrt{2q})/2$ points and $B_i$ in $(q+ \sqrt{2q})/2$ points, respectively. As before, let $\cE_j$, $\cE_k$ be two distinct elliptic quadrics of $\cP$, for some fixed indexes $j$, $k$, with $j \ne k$. Let $\cO_2$ be the set $A_j \cup B_k \cup \cQ(4, q)$. Then $\cO_2$ is a $(q+1)$-ovoid of $\cW(5, q)$.

\begin{prop}
If $\ell$ is a line of $\cW(5, q)$, then
\begin{align*}
    & |\ell \cap \cO_1| \in \left\{0, 1, 2, \frac{q}{2}+1, q+1\right\}, \\
    & |\ell \cap \cO_2| \in \left\{0,1,2,\frac{q-\sqrt{2q}}{2}+1, \frac{q+\sqrt{2q}}{2}+1, q+1 \right\}.
\end{align*}
In particular there are exactly $(q+1)(q^2+1)$ lines of $\cW(5, q)$ contained in both $\cO_1$ and $\cO_2$.
\end{prop}
\begin{proof}
Let $\ell$ be a line of $\cW(5, q)$. Since $\cO_1 \cap V^\perp = \cQ(4, q)$ and a line of $\cW(5, q)$ contained in $V^\perp$ meets $\cQ(4, q)$ in one or $q+1$ points, we have that if $\ell$ is contained in $V^\perp$, then $\ell$ has either $1$ or $q+1$ points in common with $\cO_1$. If $\ell \cap V^\perp$ is a point, let $g$ be a generator of $\cW(5, q)$ with $\ell \subset g$. Thus $g \cap V^\perp$ is a line, say $r$. If $r \subset \cQ(4, q)$, then $|g \cap (\cR_j \cup \cR_k)| = 0$ since $\cO_1 \setminus \cQ(4, q) = \cR_j \cup \cR_k$. Hence $\ell \cap \cO_1 = \ell \cap r$ and $|\ell \cap \cO_1| = 1$. If $r \cap \cQ(4, q)$ is a point $P$, then $g \cap \cR_i$, consists of $q/2$ points of a line $r_i$ of $\cE_i$, $i = 1, \dots, q/2$. Furthermore $P = r_j \cap r_k$. Therefore in this case $\ell \cap \cO_1$ has either $0$, or $1$, or $2$, or $q/2+1$ points. 

A similar proof holds for $\cO_2$. Finally, note that the lines of $\cW(5, q)$ contained in $\cO_i$, $i = 1,2$ are those of $\cQ(4, q)$.  
\end{proof}

Denote by $\cO$ and $\cO'$ the $(q+1)$-ovoid of $\cW(5, q)$ obtained in Theorem~\ref{second} and Theorem~\ref{first}. By Proposition~\ref{contained} and Proposition~\ref{prop:no_disjoint}, there are exactly $q^2+1$ and $q^2(q+1)+1$ lines of $\cW(5, q)$ contained in $\cO$ and $\cO'$, respectively. Therefore $\cO_1$, $\cO_2$, $\cO$, $\cO'$ are four distinct pairwise non-isomorphic examples of $(q+1)$-ovoids of $\cW(5, q)$, $q$ even. Taking into account Corollary~\ref{cor:disjoint} and the fact that $\cO$ contains precisely $q^2+1$ lines of $\cW(5, q)$, we infer that neither $\cO$ nor $\cO'$ consists of the points covered by a $1$-system of $\cW(5, q)$. 

\section{The $m$-ovoids of $\cW(5, 2)$}\label{W(5,2)}

Let $\cX$ be an $m$-ovoid of $\cW(5, 2)$. Since $\cW(5, 2)$ has no ovoids \cite{Thas}, we may assume that $m \in \{2, 3\}$. Let $A$ be the incidence matrix with rows indexed by points of $\cW(5, 2)$ and columns indexed by generators of $\cW(5, 2)$. Then $\v = (x_1, \dots, x_{63})$ is the characteristic vector of $\cX$ if and only if 
\begin{align*}
\v A = m \j, \label{ilp} 
\end{align*}
where $x_i \in \{0, 1\}$, $\sum_{i = 1}^{63} x_i = 9m$ and $\j$ denotes the all ones vector of size $135$. The problem of determining the entries of $\v$ can be seen as an integer linear programming problem. 
Solving it by using the mixed integer linear programming software Gurobi \cite{gurobi} we find that it is infeasible if $m = 2$ and there are exactly three pairwise non-isomorphic examples if $m = 3$. The first example arises from an elliptic quadric $\cQ^-(5, 2)$ polarizing to $\cW(5, 2)$, whereas the other two are the $3$-ovoids described in Theorem~\ref{first} and Theorem~\ref{second}. 

\smallskip
{\footnotesize
\noindent\textit{Acknowledgments.}
This work was supported by the Italian National Group for Algebraic and Geometric Structures and their Applications (GNSAGA-- INdAM).}

\end{document}